\documentclass[12pt,reqno]{amsart}

\usepackage[margin=1in,footskip=0.5in]{geometry}

\usepackage{mathtools}
\usepackage{amssymb}
\usepackage{bm}
\usepackage{amsthm}
\usepackage{amsmath}
\usepackage{thmtools}

\usepackage{graphicx}
\usepackage{tikz}
\usepackage[all]{xy}

\usepackage{enumitem}

\usepackage{xcolor}

\definecolor{myred}{rgb}{0.92,0.04,0.04}
\definecolor{myblue}{rgb}{0.14,0.36,0.90}
\definecolor{mygreen}{rgb}{0.00,0.46,0.12}
\newcommand{\red}[1]{\textcolor{myred}{#1}} 
\newcommand{\blue}[1]{\textcolor{myblue}{#1}} 
\newcommand{\green}[1]{\textcolor{mygreen}{#1}}

\usepackage[colorlinks=true,citecolor=blue,linkcolor=red,urlcolor=black,linktocpage,unicode]{hyperref}

\usetikzlibrary{patterns,decorations.pathmorphing,decorations.pathreplacing}

\theoremstyle{plain}
\numberwithin{equation}{section}

\newtheorem{thm}{Theorem}[section]
\newtheorem*{thm*}{Theorem}
\newtheorem{lem}[thm]{Lemma}
\newtheorem*{lem*}{Lemma}
\newtheorem{cor}[thm]{Corollary}

\newtheorem{defn}[thm]{Definition}

\newtheorem{ex}[thm]{Example}
\newtheorem{conj}[thm]{Conjecture}
\newtheorem{ques}[thm]{Question}



\selectfont

\begin{document}

\title{The Fujimoto Conjecture via Total Positivity}
\author{Shuhei Katsuta}
\thanks{This work was supported by JSPS KAKENHI Grant Number JP24KJ1283.}
\address{Graduate School of Mathematics, Nagoya University. Furocho, Chikusaku, Nagoya, Japan, 464-8602.}
\email{shuhei.katsuta.e6@math.nagoya-u.ac.jp}
\date{}

\begin{abstract}
  H. Fujimoto showed that for a complete minimal surface in $ \mathbb{R}^m $, if the Gauss map is non-degenerate, then it omits at most $ \frac{m(m + 1)}{2} $ hyperplanes in the complex projective space $ \mathbb{P}^{m - 1} $ in general position, and that the number $ \frac{m(m + 1)}{2} $ is best possible for all odd integers $ m \geq 3 $ and for even integers with $ 4 \leq m \leq 16 $. In this paper, we prove that the number $ \frac{m(m + 1)}{2} $ is also best possible for all even integers $ m \geq 4 $, as conjectured by Fujimoto. The main tool is a special planar network $ (\Gamma_0, \omega) $ in the theory of positive matrices.
\end{abstract}

\maketitle

\section{Introduction}
In 1990, H. Fujimoto formulated the following conjecture.
\begin{conj}[\textbf{Fujimoto conjecture}, \cite{Fujimoto_3}, p.~384; \cite{Fujimoto_2}, p.~319; \cite{Fujimoto_1}, p.~42; \cite{Fujimoto_7}, p.~198]\label{conj_Fujimoto}

  Let $ m \geq 2 $ be an even integer, and set $ t \coloneqq \frac{m}{2} $. Then the following $ 3t $ polynomials
  \begin{alignat*}{3}
     & f_i(z) \coloneqq z^{i - 1}                        & \quad & (1 \leq i \leq t),      \\
     & f_i(z) \coloneqq (z - 1)^{i - 1}                  & \quad & (t + 1 \leq i \leq 2t), \\
     & f_i(z) \coloneqq z^{i - t - 1}(z - 1)^{m - i + t} & \quad & (2t + 1 \leq i \leq 3t)
  \end{alignat*}
  are in general position. That is, any $ m $ polynomials chosen from the above are linearly independent over $ \mathbb{C} $.
\end{conj}
The aim of this paper is to prove this conjecture. Although the statement is purely algebraic, it arises from the theory of minimal surfaces. First, we explain the background of this conjecture.

Let $ S $ be a minimal surface immersed in $ \mathbb{R}^m $, and let $ \Pi $ be the set of all oriented $ 2 $-dimensional subspaces of $ \mathbb{R}^m $. It is well known that $ S $ admits the structure of a Riemann surface, and that
$ \Pi $ can be identified with the quadric
\begin{equation*}
  Q_{m - 2} \coloneqq \{(w_1 : w_2 : \cdots : w_m) \in \mathbb{P}^{m - 1} \mid w_1^2 + w_2^2 + \cdots + w_m^2 = 0\}.
\end{equation*}
The (generalized) \textbf{Gauss map} $ G : S \to Q_{m - 2} \subset \mathbb{P}^{m - 1} $ is defined by mapping a point $ p \in S $ to the tangent space $ T_p(S) \in \Pi \simeq Q_{m - 2} $, and is holomorphic. We also introduce the complex Gauss map as follows. When an open Riemann surface $ S $ is holomorphically immersed in $ \mathbb{C}^{m} $ (i.e., there exists a holomorphic curve $ f : S \to \mathbb{C}^m $ such that $ df $ vanishes nowhere), the \textbf{complex Gauss map} is defined by sending $ p \in S $ to the point in $ \mathbb{P}^{m - 1} $ corresponding to the complex tangent line $ T_p(S) $ at $ p $.
For details, see \cite{Fujimoto_3, Fujimoto_5, Fujimoto_7, Osserman, Ru_3}. The study of the generalized Gauss map from the perspective of value distribution theory (Weyl--Ahlfors theory) was started by S. S. Chern and R. Osserman \cite{Chern-Osserman}. This direction was pursued by many researchers (for the history, for example, see \cite{Fujimoto_7, Osserman, Ru_2}), and in the course of these studies, Fujimoto obtained the following theorems by establishing the modified defect relation (see \cite{Fujimoto_3, Fujimoto_4, Fujimoto_5, Fujimoto_6, Fujimoto_7}).
\begin{thm}[\cite{Fujimoto_3}, p.~366]\label{thm_gauss_map_Rm}
  Let $ S $ be a complete minimal surface immersed in $ \mathbb{R}^m $. Assume that the Gauss map $ G $ of $ S $ is non-degenerate (i.e., $ G(S) $ is not contained in any hyperplane in $ \mathbb{P}^{m - 1} $). Then $ G $ omits at most $ \frac{m(m + 1)}{2} $ hyperplanes in $ \mathbb{P}^{m - 1} $ in general position.
\end{thm}
\begin{thm}[\cite{Fujimoto_3}, p.~366]\label{thm_gauss_map_Cm}
  Let $ S $ be a complete Riemann surface holomorphically immersed in $ \mathbb{C}^m $. If $ S $ is not contained in any affine hyperplane in $ \mathbb{C}^m $, then the complex Gauss map $ G $ of $ S $ omits at most $ \frac{m(m + 1)}{2} $ hyperplanes in $ \mathbb{P}^{m - 1} $ in general position.
\end{thm}

We note that M. Ru \cite{Ru} proved that the non-degeneracy condition in Theorem~\ref{thm_gauss_map_Rm} can be removed, provided that $S$ is non-flat.

Determining whether the number $ \frac{m(m + 1)}{2} $ is best possible is an important problem. When $ m = 3 $, Theorem~\ref{thm_gauss_map_Rm} says that the generalized Gauss map $ G $ omits at most $ \frac{3(3 + 1)}{2} = 6 $ hyperplanes in $ \mathbb{P}^{2} $. Here, the number $ 6 $ is best possible. Indeed, one can construct a non-flat complete minimal surface in $ \mathbb{R}^3 $ whose generalized Gauss map omits $ 6 $ hyperplanes in $ \mathbb{P}^2 $ in general position, using the fact that there exists a non-flat complete minimal surface $ S_0 $ in $ \mathbb{R}^3 $ such that its classical Gauss map $ g : S_0 \to \mathbb{P}^1 $ omits 4 distinct points in $ \mathbb{P}^1 $ (e.g., the Scherk surface), and that $ Q_1 $ is biholomorphic to $ \mathbb{P}^1 $. (For details, see \cite{Fujimoto_2}, pp.~279--280.)
Moreover, Fujimoto also showed that the above theorems are best possible for certain other values of $ m $.

\begin{thm}[\cite{Fujimoto_1}, p.~37; \cite{Fujimoto_3}, p.~382; Y. Kawakami \cite{Kawakami}, p.~45]\label{thm_best_possible}
  The number $ \frac{m(m + 1)}{2} $ in Theorem~\ref{thm_gauss_map_Rm} (resp. Theorem~\ref{thm_gauss_map_Cm}) is best possible for all odd integers $ m \geq 3 $ and even integers $ m $ with $ 4 \leq m \leq 16 $. That is, for each such $ m $, there exists a complete (which may be chosen to be pseudo-algebraic; cf. \cite{Kawakami}) minimal surface immersed in $ \mathbb{R}^m $ (resp. a complete Riemann surface immersed in $ \mathbb{C}^m $) whose Gauss map is non-degenerate (resp. not contained in any affine hyperplane) and omits $ \frac{m(m + 1)}{2} $ hyperplanes in $ \mathbb{P}^{m - 1} $ in general position.
\end{thm}

A key ingredient in the construction of a best possible example for Theorem~\ref{thm_best_possible} for each odd integer $ m $ is an algebraic lemma (\cite{Fujimoto_3}, pp.~382--383, Lemma~6.2) that is used to determine the desired surface and the coefficients of the omitted hyperplanes in $ \mathbb{P}^{m - 1} $.
Fujimoto conjectured that a similar construction would also work to obtain best possible examples of the above theorems for even integers $ m $. In order to prove an analogue of this algebraic lemma for the even-dimensional case (given in the appendix), he proposed the Fujimoto conjecture (Conjecture~\ref{conj_Fujimoto}) and verified by computer that it is true for all even integers $ m \leq 16 $.

In this paper, we show that the conjecture holds for all even integers $ m \geq 2 $.
\begin{thm}[\textbf{Main Theorem}]\label{thm_main_theorem}
  The Fujimoto conjecture (Conjecture~\ref{conj_Fujimoto}) is true for all even integers $ m \geq 2 $.
\end{thm}
Therefore, the following expectation of Fujimoto is fully justified.
\begin{cor}[\cite{Fujimoto_3}, p.~384; \cite{Fujimoto_1}, pp.~42--43; \cite{Fujimoto_7}, pp.~198--199]\label{cor_best_possible}
  The number $ \frac{m(m + 1)}{2} $ in Theorems~\ref{thm_gauss_map_Rm} and~\ref{thm_gauss_map_Cm} is best possible for all integers $ m \geq 3 $.
  Moreover, an example for Theorem~\ref{thm_gauss_map_Rm} can be chosen to be a pseudo-algebraic minimal surface in $ \mathbb{R}^m $.
\end{cor}
\begin{proof}
  We can construct these examples in a similar way to that in \cite{Fujimoto_3}, p.~382 and p.~384, \cite{Fujimoto_1}, pp.~39--42, and \cite{Kawakami}, pp.~45--47. See the appendix for details.
\end{proof}

For the remaining problems, see the Questions section.
As noted by Osserman and Ru \cite{Osserman-Ru}, p.~579, the following result, which is stronger than Theorem~\ref{thm_gauss_map_Rm}, is also best possible.
\begin{cor}[\cite{Osserman-Ru}, p.~278]
  Let $ S $ be a minimal surface immersed in $ \mathbb{R}^m $. For $ p \in S $, we denote its Gauss curvature at $ p $ by $ K(p) $, and the geodesic distance from $ p $ to the boundary of $ S $ by $ d(p) $. If the Gauss map of $ S $ omits more than $ \frac{m(m + 1)}{2} $ hyperplanes in $ \mathbb{P}^{m - 1} $ in general position, then there exists a constant $ C $ (which depends only on the hyperplanes) such that
  \begin{equation*}
    |K(p)|^{\frac{1}{2}}d(p) \leq C.
  \end{equation*}
  Moreover, the bound $ \frac{m(m + 1)}{2} $ is optimal.
\end{cor}
\begin{proof}
  Here, we only verify the sharpness of this result.
  For each $ m \geq 3 $, let $ S $ be a best possible example of Corollary~\ref{cor_best_possible} in $ \mathbb{R}^m $. Suppose that the above inequality holds for $ S $. Since $ d(p) = \infty $ for all $ p \in S $, we have $ K \equiv 0 $. This implies that $ S $ must lie on a plane (see, for example, \cite{Fujimoto_7}, p.~15, Proposition~2.13), which is a contradiction.
\end{proof}

Besides the introduction, this paper contains two sections, some questions, and an appendix.
From the viewpoint of the theory of positive matrices, the Fujimoto conjecture can be reduced to the construction of a certain weighted directed graph (called a \emph{planar network}). After reviewing the basic parts of the theory of positive matrices needed in this paper (Section~\ref{sec_2}), we construct a planar network with nonnegative weights representing the matrix $ \mathsf{M} $ corresponding to the $ 3t $ polynomials $ f_i(z) $ (Section~\ref{sec_3}). In Fujimoto's paper \cite{Fujimoto_3}, the details of the proof of the existence of the auxiliary polynomials used in the construction of best possible examples for Theorem~\ref{thm_best_possible} for each even integer $ m $ are omitted to avoid repetition. We therefore provide a proof of this omitted part and construct a best possible example for Theorem~\ref{thm_gauss_map_Rm} in the appendix, following his approach.

\section{Totally positive matrix and totally nonnegative matrix}\label{sec_2}
In this section, we recall some basic results of the theory of totally positive and totally nonnegative matrices. This theory fits very well with the situation we are considering. We refer the reader to \cite{Fomin-Zelevinsky} by S. Fomin and A. Zelevinsky, which provides a detailed account of the subject.

Let $ M $ be an $ n \times n $ matrix, and $ I, J \subset \{1, 2, \ldots, n\} $ with $ |I| = |J| $. We denote  by $ M_{I, J} $ the submatrix obtained from $ M $ by deleting the rows indexed by $ \{1, 2, \ldots, n\} \setminus I $ and the columns indexed by $ \{1, 2, \ldots, n\}\setminus J $. The determinant of $ M_{I, J} $ is called the minor of $ M $ corresponding to $ I $ and $ J $, and is denoted by $ \Delta_{I, J} $.
\begin{defn}[Totally positive matrix, totally nonnegative matrix]
  Let $ M $ be a (real) matrix. The matrix $ M $ is \textbf{totally positive} (resp. \textbf{totally nonnegative}) if all minors are positive (resp. nonnegative).
\end{defn}
\begin{defn}[Planar network]
  Let $ \Gamma $ be a directed acyclic planar graph. We denote the set of edges (resp. vertices) of $ \Gamma $ by $ E(\Gamma) $ (resp. $ V(\Gamma) $). A function $ \omega : E(\Gamma) \to \mathbb{R} $ is called an edge weight function, and the pair $ (\Gamma, \omega) $ is called a \textbf{planar network}. For each integer $ n \geq 1 $, we define a special planar network $ (\Gamma_0, \omega = (l_{i, j}, m_k, r_{i, j})) $ as in Figure~\ref{fig_Gamma_0}, where $ \Gamma_0 $ is directed from left to right, and $ l_{i, j} $, $ m_k $, and $ r_{i, j} \ ( 1 \leq i, j \leq n - 1, 2 \leq i + j \leq n, 1 \leq k \leq n ) $ are real numbers assigned to the red, green, and blue edges, respectively. Each black edge is always assigned the weight $ 1 $.
  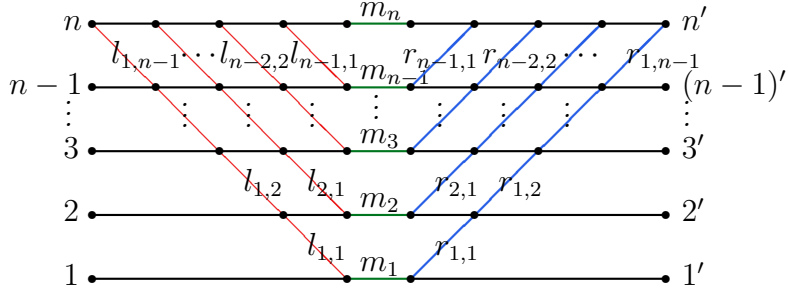
\begin{figure}[ht]
    \setlength{\unitlength}{1.2pt}
    \begin{center}
      \begin{picture}(180,85)(0,-5)
        \thicklines

        \put(0,0){\line(1,0){80}}
        \put(0,20){\line(1,0){80}}
        \put(0,40){\line(1,0){80}}
        \put(0,60){\line(1,0){80}}
        \put(0,80){\line(1,0){80}}

        \green{
          \put(80,0){\line(1,0){20}}
          \put(80,20){\line(1,0){20}}
          \put(80,40){\line(1,0){20}}
          \put(80,60){\line(1,0){20}}
          \put(80,80){\line(1,0){20}}
        }

        \put(84,2.5){$ m_1 $}
        \put(84,22.5){$ m_2 $}
        \put(84,42.5){$ m_3 $}
        \put(87.5,50){$ \vdots $}
        \put(84,62.5){$ m_{n - 1} $}
        \put(84,82.5){$ m_n $}

        \put(100,0){\line(1,0){80}}
        \put(100,20){\line(1,0){80}}
        \put(100,40){\line(1,0){80}}
        \put(100,60){\line(1,0){80}}
        \put(100,80){\line(1,0){80}}

        \blue{
          \put(100,0){\line(1,1){80}}
          \put(100,20){\line(1,1){60}}
          \put(100,40){\line(1,1){40}}
          \put(100,60){\line(1,1){20}}
        }

        \put(107.5,7.5){$ r_{1, 1} $}
        \put(107.5,27.5){$ r_{2, 1} $}
        \put(107.5,47.5){$ \vdots $}
        \put(97.5,67.5){$ r_{n - 1, 1} $}

        \put(127.5,27.5){$ r_{1, 2} $}
        \put(127.5,47.5){$ \vdots $}
        \put(122.5,67.5){$ r_{n - 2, 2} $}

        \put(147.5,47.5){$ \vdots $}
        \put(147.5,67.5){$ \cdots $}

        \put(167.5,67.5){$ r_{1, n - 1} $}

        \thinlines
        \red{
          \put(60,80){\line(1,-1){20}}
          \put(40,80){\line(1,-1){40}}
          \put(20,80){\line(1,-1){60}}
          \put(0,80){\line(1,-1){80}}
        }

        \put(67.5,7.5){$ l_{1, 1} $}
        \put(67.5,27.5){$ l_{2, 1} $}
        \put(67.5,47.5){$ \vdots $}
        \put(62,67.5){$ l_{n - 1, 1} $}

        \put(47.5,27.5){$ l_{1, 2} $}
        \put(47.5,47.5){$ \vdots $}
        \put(40,67.5){$ l_{n - 2, 2} $}

        \put(27.5,47.5){$ \vdots $}
        \put(27.5,67.5){$ \cdots $}

        \put(6,67.5){$ l_{1, n - 1} $}

        \put(185,-2){$ 1' $}
        \put(185,18){$ 2' $}
        \put(185,38){$ 3' $}
        \put(185,48){$ \vdots $}
        \put(185,58){$ (n - 1)' $}
        \put(185,78){$ n' $}

        \put(-9,-2){$ 1 $}
        \put(-9,18){$ 2 $}
        \put(-9,38){$ 3 $}
        \put(-9,48){$ \vdots $}
        \put(-26,58){$ n - 1 $}
        \put(-9,78){$ n $}

        \put(0,0){\circle*{2.5}}
        \put(0,20){\circle*{2.5}}
        \put(0,40){\circle*{2.5}}
        \put(0,60){\circle*{2.5}}
        \put(0,80){\circle*{2.5}}

        \put(20,60){\circle*{2.5}}
        \put(20,80){\circle*{2.5}}

        \put(40,40){\circle*{2.5}}
        \put(40,60){\circle*{2.5}}
        \put(40,80){\circle*{2.5}}

        \put(60,20){\circle*{2.5}}
        \put(60,40){\circle*{2.5}}
        \put(60,60){\circle*{2.5}}
        \put(60,80){\circle*{2.5}}

        \put(80,0){\circle*{2.5}}
        \put(80,20){\circle*{2.5}}
        \put(80,40){\circle*{2.5}}
        \put(80,60){\circle*{2.5}}
        \put(80,80){\circle*{2.5}}

        \put(100,0){\circle*{2.5}}
        \put(100,20){\circle*{2.5}}
        \put(100,40){\circle*{2.5}}
        \put(100,60){\circle*{2.5}}
        \put(100,80){\circle*{2.5}}

        \put(120,20){\circle*{2.5}}
        \put(120,40){\circle*{2.5}}
        \put(120,60){\circle*{2.5}}
        \put(120,80){\circle*{2.5}}

        \put(140,40){\circle*{2.5}}
        \put(140,60){\circle*{2.5}}
        \put(140,80){\circle*{2.5}}

        \put(160,60){\circle*{2.5}}
        \put(160,80){\circle*{2.5}}

        \put(180,0){\circle*{2.5}}
        \put(180,20){\circle*{2.5}}
        \put(180,40){\circle*{2.5}}
        \put(180,60){\circle*{2.5}}
        \put(180,80){\circle*{2.5}}
      \end{picture}
    \end{center}
    \caption{Planar network $(\Gamma_0, (l_{i, j}, m_k, r_{i, j}))$}
    \label{fig_Gamma_0}
  \end{figure}
\end{defn}
\begin{defn}[Weight matrix]
  Let $ (\Gamma, \omega) $ be a finite planar network with subsets $ I = \{v_1, v_2, \ldots, v_n\} $ and $ J = \{w_1, w_2, \ldots, w_n\} $ of the set $ V(\Gamma) $. Then the $ (i, j) $-entry of its \textbf{weight matrix} $ x(\Gamma, \omega) $ is defined as the sum of the weights of all paths from $ v_i $ to $ w_j $, where the weight of a path means the product of all the weights on the path.
\end{defn}
\begin{thm}[\textbf{Lindstr\"om--Gessel--Viennot lemma}, B. Lindstr\"om \cite{Lindstrom}, pp.~87--88, I. Geesel and G. X. Viennot \cite{Gessel-Viennot}, pp.~302--303]\label{thm_LGV}
  Let $ (\Gamma, \omega) $ be a finite planar network with subsets $ I $ and $ J $ of the set $ V(\Gamma) $ satisfying the following conditions:
  \begin{itemize}
    \item $ n \coloneqq |I| = |J| $.
    \item Each vertex $ v $ in $ I $ (resp. $ J $) is labelled by an integer from $ 1 $ to $ n $: $ I = \{v_1, v_2, \ldots, v_n\} $ (resp. $ J = \{w_1, w_2, \ldots, w_n\} $) so that if $ i_1 > i_2 $ and $ j_1 < j_2 $ $ (1 \leq i_1, i_2, j_1, j_2 \leq n) $, then any two paths from $ v_{i_1} $ to $ w_{j_1} $ and from $ v_{i_2} $ to $ w_{j_2} $ intersect.
  \end{itemize}
  Then each minor $ \Delta_{I', J'} \ (I' \subset I, J' \subset J, |I'| = |J'|) $ of the weight matrix $ x(\Gamma, \omega) $ equals the sum of the weights of all collections of non-intersecting paths from $ I' $ to $ J' $, where the weight of a collection of non-intersecting paths means the product of all the weights on these paths.
\end{thm}
\begin{ex}
  Consider the $ n' \times n' $ grid $ \Gamma $. Let $ V(\Gamma) = \{(i, j) \in \mathbb{Z}^2 \mid 1 \leq i, j \leq n' + 1\} $ and let $ E(\Gamma) = \{((i, j), (i', j')) \in V(\Gamma) \times V(\Gamma) \mid (i', j') - (i, j) = (1, 0) \ \text{or} \ (0, 1)\} $. Then $ (\Gamma, \mathsf{1}) $ is a planar network, where $ \mathsf{1} $ denotes the constant weight function with value $ 1 $. Fix an integer $ 1 \leq n \leq n' + 1 $, and let
  \begin{equation*}
    I = \{v_i = (1, i) \mid 1 \leq i \leq n\}, \quad J = \{w_i = (n' - i + 2, n' + 1) \mid 1 \leq i \leq n\}.
  \end{equation*}
  Then $ I $ and $ J $ satisfy the conditions of Theorem~\ref{thm_LGV}, and it is easy to see that
  \begin{equation*}
    x(\Gamma, \mathsf{1}) = \left(\binom{2n' - i - j + 2}{n' - i + 1}\right)_{1 \leq i, j \leq n}.
  \end{equation*}
  Therefore, for $ I' = \{v_{i_1}, \ldots, v_{i_s}\} \subset I $ $ (1 \leq i_1 < \cdots < i_s \leq n)$ and $ J' = \{w_{j_1}, \ldots, w_{j_s}\} \subset J $ $ (1 \leq j_1 < \cdots < j_s \leq n) $, we have
  \begin{equation*}
    |\{\text{collections of non-intersecting paths from $ I' $ to $ J' $} \ \text{in} \  \Gamma \}| = \det \left(\binom{2n' - i_k - j_l + 2}{n' - i_k + 1}\right)_{1 \leq k, l \leq s}.
  \end{equation*}
\end{ex}
The Lindstr\"om--Gessel--Viennot lemma (Theorem~\ref{thm_LGV}) implies the following.
\begin{thm}[\cite{Fomin-Zelevinsky}, p.~24]
  Let $ (\Gamma, \omega) $ be a finite planar network, and let $ I $, $ J $ be subsets of the set $ V(\Gamma) $ satisfying the conditions of Theorem~\ref{thm_LGV}. If $ \omega $ is a nonnegative edge weight function, then $ x(\Gamma, \omega) $ is totally nonnegative. If $ \omega $ is a positive edge weight function and there exists at least one collection of non-intersecting paths from $ I' $ to $ J' $, for all $ I' \subset I $, $ J' \subset J $ with $ |I'| = |J'| $, then $ x(\Gamma, \omega) $ is totally positive.
\end{thm}
The importance of the special planar network $ (\Gamma_0, \omega) $ lies in the following theorem.
\begin{thm}[\cite{Fomin-Zelevinsky}, p.~25]\label{thm_bij}
  For any totally positive matrix $ M $, there exists a positive edge weight function $ \omega = (l_{i, j}, m_k, r_{i, j}) $ on $ \Gamma_0 $ such that
  \begin{equation*}
    x(\Gamma_0, \omega) = M.
  \end{equation*}
  More strongly, the map $ \omega \mapsto x(\Gamma_0, \omega) $ establishes a bijection between the set of all positive edge weight functions $ \omega = (l_{i, j}, m_k, r_{i, j})_{1 \leq i, j \leq n - 1, \ 2 \leq i + j \leq n, \ 1 \leq k \leq n} $ on $ \Gamma_0 $ and the set of all totally positive $ n \times n $ matrices.
\end{thm}
Concerning totally nonnegative matrices, A. Whitney \cite{Whitney} obtained the following result.
\begin{thm}[\textbf{Whitney's theorem}, \cite{Whitney}, pp.~88--89]\label{thm_Whitney}
  Every invertible totally nonnegative matrix is obtained as the limit of a sequence of totally positive matrices.
\end{thm}

\section{Proof of the Fujimoto conjecture}\label{sec_3}
\subsection{Observation}\label{subsec_observation}
Before proceeding to the proof, we see some examples.
\begin{itemize}
  \item $ m = 2 \ (t = 1)$. Then, $ f_1(z) = 1 $, $ f_2(z) = z - 1 $, and $ f_3(z) = z $. These polynomials can be represented by a $ 3t \times m = 3 \times 2 $ matrix by expanding $ f_i(z) $ with respect to monomials $ \{z^i\}_{0 \leq i \leq 2t - 1} = \{z^i\}_{0 \leq i \leq 1} $:
        \begin{equation*}
          \begin{pmatrix}
            1     \\
            z - 1 \\
            z     \\
          \end{pmatrix} =
          \begin{pmatrix}
            1  & 0 \\
            -1 & 1 \\
            0  & 1 \\
          \end{pmatrix}
          \begin{pmatrix}
            1 \\
            z \\
          \end{pmatrix}.
        \end{equation*}
  \item $ m = 4 \ (t = 2)$. Then, $ f_1(z) = 1 $, $ f_2(z) = z $, $ f_3(z) = (z - 1)^2 $, $ f_4(z) = (z - 1)^3 $, $ f_5(z) = z^2(z - 1) $, and $ f_6(z) = z^3 $. These polynomials can be represented as
        \begin{equation*}
          \begin{pmatrix}
            1          \\
            z          \\
            (z - 1)^2  \\
            (z - 1)^3  \\
            z^2(z - 1) \\
            z^3        \\
          \end{pmatrix} =
          \begin{pmatrix}
            1  & 0  & 0  & 0 \\
            0  & 1  & 0  & 0 \\
            1  & -2 & 1  & 0 \\
            -1 & 3  & -3 & 1 \\
            0  & 0  & -1 & 1 \\
            0  & 0  & 0  & 1 \\
          \end{pmatrix}
          \begin{pmatrix}
            1   \\
            z   \\
            z^2 \\
            z^3
          \end{pmatrix}.
        \end{equation*}
  \item $ m = 6 \ (t = 3)$. Then, $ f_1(z) = 1 $, $ f_2(z) = z $, $ f_3(z) = z^2 $, $ f_4(z) = (z - 1)^3 $, $ f_5(z) = (z - 1)^4 $, $ f_6(z) = (z - 1)^5 $, $ f_7(z) = z^3(z - 1)^2 $, $ f_8(z) = z^4(z - 1) $, and $ f_9(z) = z^5 $. These polynomials can be represented as
        \begin{equation*}
          \begin{pmatrix}
            1            \\
            z            \\
            z^2          \\
            (z - 1)^3    \\
            (z - 1)^4    \\
            (z - 1)^5    \\
            z^3(z - 1)^2 \\
            z^4(z - 1)   \\
            z^5          \\
          \end{pmatrix} =
          \begin{pmatrix}
            1  & 0  & 0   & 0  & 0  & 0 \\
            0  & 1  & 0   & 0  & 0  & 0 \\
            0  & 0  & 1   & 0  & 0  & 0 \\
            -1 & 3  & -3  & 1  & 0  & 0 \\
            1  & -4 & 6   & -4 & 1  & 0 \\
            -1 & 5  & -10 & 10 & -5 & 1 \\
            0  & 0  & 0   & 1  & -2 & 1 \\
            0  & 0  & 0   & 0  & -1 & 1 \\
            0  & 0  & 0   & 0  & 0  & 1 \\
          \end{pmatrix}
          \begin{pmatrix}
            1   \\
            z   \\
            z^2 \\
            z^3 \\
            z^4 \\
            z^5 \\
          \end{pmatrix}.
        \end{equation*}
\end{itemize}
Note that the $ m \times m $ matrices obtained by taking the rows from the $ (t + 1) $-th through the $ 3t $-th rows of the above coefficient matrices and replacing each entry by its absolute value are represented in the form $ x(\Gamma_0, \omega) $. Indeed, the corresponding planar networks are given as shown in Figure~\ref{Weighted planar network}.
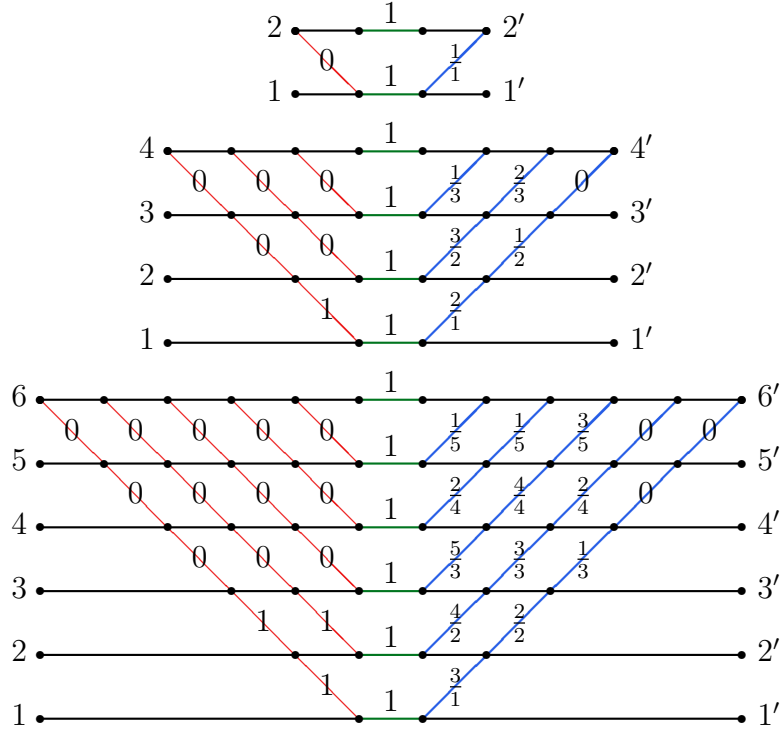
\begin{figure}[ht]
  \setlength{\unitlength}{1.2pt}
  \begin{center}
    \begin{picture}(180,40)(0,-10)
      \thicklines

      \put(60,0){\line(1,0){20}}
      \put(60,20){\line(1,0){20}}

      \green{
        \put(80,0){\line(1,0){20}}
        \put(80,20){\line(1,0){20}}
      }

      \put(87.5,2.5){$ 1 $}
      \put(87.5,22.5){$ 1 $}

      \put(100,0){\line(1,0){20}}
      \put(100,20){\line(1,0){20}}

      \blue{
        \put(100,0){\line(1,1){20}}
      }

      \put(107.5,7.5){$ \frac{1}{1} $}

      \thinlines
      \red{
        \put(60,20){\line(1,-1){20}}
      }

      \put(67.5,7.5){$ 0 $}

      \thicklines

      \put(125,-2){$ 1' $}
      \put(125,18){$ 2' $}

      \put(51,-2){$ 1 $}
      \put(51,18){$ 2 $}

      \put(60,0){\circle*{2.5}}
      \put(60,20){\circle*{2.5}}

      \put(60,20){\circle*{2.5}}

      \put(80,0){\circle*{2.5}}
      \put(80,20){\circle*{2.5}}

      \put(100,0){\circle*{2.5}}
      \put(100,20){\circle*{2.5}}

      \put(120,20){\circle*{2.5}}

      \put(120,0){\circle*{2.5}}
      \put(120,20){\circle*{2.5}}
    \end{picture}
  \end{center}
  \begin{center}
    \begin{picture}(180,85)(0,-20)
      \thicklines

      \put(20,0){\line(1,0){60}}
      \put(20,20){\line(1,0){60}}
      \put(20,40){\line(1,0){60}}
      \put(20,60){\line(1,0){60}}

      \green{
        \put(80,0){\line(1,0){20}}
        \put(80,20){\line(1,0){20}}
        \put(80,40){\line(1,0){20}}
        \put(80,60){\line(1,0){20}}
      }

      \put(87.5,2.5){$ 1 $}
      \put(87.5,22.5){$ 1 $}
      \put(87.5,42.5){$ 1 $}
      \put(87.5,62.5){$ 1 $}

      \put(100,0){\line(1,0){60}}
      \put(100,20){\line(1,0){60}}
      \put(100,40){\line(1,0){60}}
      \put(100,60){\line(1,0){60}}

      \blue{
        \put(100,0){\line(1,1){60}}
        \put(100,20){\line(1,1){40}}
        \put(100,40){\line(1,1){20}}
      }

      \put(107.5,7.5){$ \frac{2}{1} $}
      \put(107.5,27.5){$ \frac{3}{2} $}
      \put(107.5,47.5){$ \frac{1}{3} $}

      \put(127.5,27.5){$ \frac{1}{2} $}
      \put(127.5,47.5){$ \frac{2}{3} $}

      \put(147.5,47.5){$ 0 $}

      \thinlines
      \red{
        \put(60,60){\line(1,-1){20}}
        \put(40,60){\line(1,-1){40}}
        \put(20,60){\line(1,-1){60}}
      }

      \put(67.5,7.5){$ 1 $}
      \put(67.5,27.5){$ 0 $}
      \put(67.5,47.5){$ 0 $}

      \put(47.5,27.5){$ 0 $}
      \put(47.5,47.5){$ 0 $}

      \put(27.5,47.5){$ 0 $}

      \thicklines

      \put(165,-2){$ 1' $}
      \put(165,18){$ 2' $}
      \put(165,38){$ 3' $}
      \put(165,58){$ 4' $}

      \put(11,-2){$ 1 $}
      \put(11,18){$ 2 $}
      \put(11,38){$ 3 $}
      \put(11,58){$ 4 $}

      \put(20,0){\circle*{2.5}}
      \put(20,20){\circle*{2.5}}
      \put(20,40){\circle*{2.5}}
      \put(20,60){\circle*{2.5}}

      \put(20,60){\circle*{2.5}}

      \put(40,40){\circle*{2.5}}
      \put(40,60){\circle*{2.5}}

      \put(60,20){\circle*{2.5}}
      \put(60,40){\circle*{2.5}}
      \put(60,60){\circle*{2.5}}

      \put(80,0){\circle*{2.5}}
      \put(80,20){\circle*{2.5}}
      \put(80,40){\circle*{2.5}}
      \put(80,60){\circle*{2.5}}

      \put(100,0){\circle*{2.5}}
      \put(100,20){\circle*{2.5}}
      \put(100,40){\circle*{2.5}}
      \put(100,60){\circle*{2.5}}

      \put(120,20){\circle*{2.5}}
      \put(120,40){\circle*{2.5}}
      \put(120,60){\circle*{2.5}}

      \put(140,40){\circle*{2.5}}
      \put(140,60){\circle*{2.5}}

      \put(160,60){\circle*{2.5}}

      \put(160,0){\circle*{2.5}}
      \put(160,20){\circle*{2.5}}
      \put(160,40){\circle*{2.5}}
      \put(160,60){\circle*{2.5}}
    \end{picture}
  \end{center}
  \begin{center}
    \begin{picture}(180,85)(0,10)
      \thicklines

      \put(-20,0){\line(1,0){100}}
      \put(-20,20){\line(1,0){100}}
      \put(-20,40){\line(1,0){100}}
      \put(-20,60){\line(1,0){100}}
      \put(-20,80){\line(1,0){100}}
      \put(-20,100){\line(1,0){100}}

      \green{
        \put(80,0){\line(1,0){20}}
        \put(80,20){\line(1,0){20}}
        \put(80,40){\line(1,0){20}}
        \put(80,60){\line(1,0){20}}
        \put(80,80){\line(1,0){20}}
        \put(80,100){\line(1,0){20}}
      }

      \put(87.5,2.5){$ 1 $}
      \put(87.5,22.5){$ 1 $}
      \put(87.5,42.5){$ 1 $}
      \put(87.5,62.5){$ 1 $}
      \put(87.5,82.5){$ 1 $}
      \put(87.5,102.5){$ 1 $}

      \put(100,0){\line(1,0){100}}
      \put(100,20){\line(1,0){100}}
      \put(100,40){\line(1,0){100}}
      \put(100,60){\line(1,0){100}}
      \put(100,80){\line(1,0){100}}
      \put(100,100){\line(1,0){100}}

      \blue{
        \put(100,0){\line(1,1){100}}
        \put(100,20){\line(1,1){80}}
        \put(100,40){\line(1,1){60}}
        \put(100,60){\line(1,1){40}}
        \put(100,80){\line(1,1){20}}
      }

      \put(107.5,7.5){$ \frac{3}{1} $}
      \put(107.5,27.5){$ \frac{4}{2} $}
      \put(107.5,47.5){$ \frac{5}{3} $}
      \put(107.5,67.5){$ \frac{2}{4} $}
      \put(107.5,87.5){$ \frac{1}{5} $}

      \put(127.5,27.5){$ \frac{2}{2} $}
      \put(127.5,47.5){$ \frac{3}{3} $}
      \put(127.5,67.5){$ \frac{4}{4} $}
      \put(127.5,87.5){$ \frac{1}{5} $}

      \put(147.5,47.5){$ \frac{1}{3} $}
      \put(147.5,67.5){$ \frac{2}{4} $}
      \put(147.5,87.5){$ \frac{3}{5} $}

      \put(167.5,67.5){$ 0 $}
      \put(167.5,87.5){$ 0 $}

      \put(187.5,87.5){$ 0 $}

      \thinlines
      \red{
        \put(60,100){\line(1,-1){20}}
        \put(40,100){\line(1,-1){40}}
        \put(20,100){\line(1,-1){60}}
        \put(0,100){\line(1,-1){80}}
        \put(-20,100){\line(1,-1){100}}
      }

      \put(67.5,7.5){$ 1 $}
      \put(67.5,27.5){$ 1 $}
      \put(67.5,47.5){$ 0 $}
      \put(67.5,67.5){$ 0 $}
      \put(67.5,87.5){$ 0 $}

      \put(47.5,27.5){$ 1 $}
      \put(47.5,47.5){$ 0 $}
      \put(47.5,67.5){$ 0 $}
      \put(47.5,87.5){$ 0 $}

      \put(27.5,47.5){$ 0 $}
      \put(27.5,67.5){$ 0 $}
      \put(27.5,87.5){$ 0 $}

      \put(7.5,67.5){$ 0 $}
      \put(7.5,87.5){$ 0 $}

      \put(-12.5,87.5){$ 0 $}

      \thicklines

      \put(205,-2){$ 1' $}
      \put(205,18){$ 2' $}
      \put(205,38){$ 3' $}
      \put(205,58){$ 4' $}
      \put(205,78){$ 5' $}
      \put(205,98){$ 6' $}

      \put(-29,-2){$ 1 $}
      \put(-29,18){$ 2 $}
      \put(-29,38){$ 3 $}
      \put(-29,58){$ 4 $}
      \put(-29,78){$ 5 $}
      \put(-29,98){$ 6 $}

      \put(-20,0){\circle*{2.5}}
      \put(-20,20){\circle*{2.5}}
      \put(-20,40){\circle*{2.5}}
      \put(-20,60){\circle*{2.5}}
      \put(-20,80){\circle*{2.5}}
      \put(-20,100){\circle*{2.5}}

      \put(-20,100){\circle*{2.5}}

      \put(0,80){\circle*{2.5}}
      \put(0,100){\circle*{2.5}}

      \put(20,60){\circle*{2.5}}
      \put(20,80){\circle*{2.5}}
      \put(20,100){\circle*{2.5}}

      \put(40,40){\circle*{2.5}}
      \put(40,60){\circle*{2.5}}
      \put(40,80){\circle*{2.5}}
      \put(40,100){\circle*{2.5}}

      \put(60,20){\circle*{2.5}}
      \put(60,40){\circle*{2.5}}
      \put(60,60){\circle*{2.5}}
      \put(60,80){\circle*{2.5}}
      \put(60,100){\circle*{2.5}}

      \put(80,0){\circle*{2.5}}
      \put(80,20){\circle*{2.5}}
      \put(80,40){\circle*{2.5}}
      \put(80,60){\circle*{2.5}}
      \put(80,80){\circle*{2.5}}
      \put(80,100){\circle*{2.5}}

      \put(100,0){\circle*{2.5}}
      \put(100,20){\circle*{2.5}}
      \put(100,40){\circle*{2.5}}
      \put(100,60){\circle*{2.5}}
      \put(100,80){\circle*{2.5}}
      \put(100,100){\circle*{2.5}}

      \put(120,20){\circle*{2.5}}
      \put(120,40){\circle*{2.5}}
      \put(120,60){\circle*{2.5}}
      \put(120,80){\circle*{2.5}}
      \put(120,100){\circle*{2.5}}

      \put(140,40){\circle*{2.5}}
      \put(140,60){\circle*{2.5}}
      \put(140,80){\circle*{2.5}}
      \put(140,100){\circle*{2.5}}

      \put(160,60){\circle*{2.5}}
      \put(160,80){\circle*{2.5}}
      \put(160,100){\circle*{2.5}}

      \put(180,80){\circle*{2.5}}
      \put(180,100){\circle*{2.5}}

      \put(200,100){\circle*{2.5}}

      \put(200,0){\circle*{2.5}}
      \put(200,20){\circle*{2.5}}
      \put(200,40){\circle*{2.5}}
      \put(200,60){\circle*{2.5}}
      \put(200,80){\circle*{2.5}}
      \put(200,100){\circle*{2.5}}
    \end{picture}
  \end{center}
  \caption{Weighted planar networks}
  \label{Weighted planar network}
\end{figure}

In general, for each $ m $, the coefficient matrix $ \mathsf{M}' $ associated with the polynomials $ f_i(z) \ (1 \leq i \leq 3t) $ is given by
\begin{equation*}
  \mathsf{M}' \coloneqq
  \begin{pmatrix}
    \mathrm{D}((-1)^{i})_{1 \leq i \leq t} & O                                          & O   \\
    O                                      & \mathrm{D}((-1)^{t + i})_{1 \leq i \leq t} & O   \\
    O                                      & O                                          & I_t
  \end{pmatrix}
  \begin{pmatrix}
    I_t & O   \\
    M_1 & M_2 \\
    O   & M_3 \\
  \end{pmatrix}
  \begin{pmatrix}
    \mathrm{D}((-1)^{i})_{1 \leq i \leq t} & O                                          \\
    O                                      & \mathrm{D}((-1)^{t + i})_{1 \leq i \leq t} \\
  \end{pmatrix},
\end{equation*}
where
\begin{align*}
  M_1 & \coloneqq \left(\binom{t + i - 1}{j - 1}\right)_{1 \leq i, j \leq t},     \\
  M_2 & \coloneqq \left(\binom{t + i - 1}{t + j - 1}\right)_{1 \leq i, j \leq t}, \\
  M_3 & \coloneqq \left(\binom{t - i}{j - i}\right)_{1 \leq i, j \leq t},
\end{align*}
and $ \mathrm{D}(d_i)_{1 \leq i \leq t} $ is a diagonal matrix whose $ (i, i) $-entry is $ d_i $.
\begin{lem}\label{lem_det}
  \begin{equation*}
    \det M_1 = \det M_2 = \det M_3 = 1.
  \end{equation*}
\end{lem}
\begin{proof}
  It suffices to show $ \det M_1 = 1 $. Since $ \binom{t + i - 1}{j - 1} = \sum_{k = 0}^{t - 1}\binom{t + i - 2}{k}\binom{1}{j - 1 - k} $ by Vandermonde's convolution formula, we have
  \begin{equation*}
    M_1 = \left(\binom{t + i - 2}{j - 1}\right)_{1 \leq i, j \leq t} \cdot \left(\binom{1}{j - i}\right)_{1 \leq i, j \leq t}
  \end{equation*}
  and hence $ \det M_1 = \det \left(\binom{t + i - 2}{j - 1}\right)_{1 \leq i, j \leq t} = \cdots = \det \left(\binom{t + i - (t + 1)}{j - 1}\right)_{1 \leq i, j \leq t} = 1 $.
\end{proof}
We define an $ m \times m $ matrix $ \mathsf{M} $ by
\begin{equation*}
  \mathsf{M} \coloneqq
  \begin{pmatrix}
    M_1 & M_2 \\
    O   & M_3 \\
  \end{pmatrix}.
\end{equation*}
The above examples suggest that $ \mathsf{M} $ is totally nonnegative. Since $ \mathsf{M} $ is invertible by Lemma~\ref{lem_det}, it is reasonable to expect that $ \mathsf{M} $ is represented in the form $ x(\Gamma_0, \omega) $ for a nonnegative edge weight function $ \omega $ in view of Theorems~\ref{thm_bij} and~\ref{thm_Whitney}. In the next subsection, we give a concrete edge weight function suggested by the above examples.

\subsection{Construction of $ (\Gamma_0, \omega_0) $}\label{subsec_construction_graph}
To prove the main theorem, we assign a nonnegative edge weight function $ \omega_0 $ on the graph $ \Gamma_0 $ corresponding to the $ m \times m $ matrix $ \mathsf{M} $.
As suggested by the above observations, the following defines the desired edge weight function $ \omega_0 = (l_{i, j}, m_{k},  r_{i, j}) $:
\begin{equation*}
  r_{i, j} \coloneqq
  \begin{cases}
    \frac{i + (t - j)}{i + j - 1}     & (1 \leq i, j \leq t), \\
    \frac{m - (i + j) + 1}{i + j - 1} & (i > t),              \\
    0                                 & (j > t),
  \end{cases}
\end{equation*}
\begin{equation*}
  m_k \coloneqq 1 \quad (1 \leq k \leq m),
\end{equation*}
and
\begin{equation*}
  l_{i, j} \coloneqq
  \begin{cases}
    1 & (i + j \leq t), \\
    0 & (i + j > t),
  \end{cases}
\end{equation*}
where $ i $ and $ j $ range over all integers $ 1 \leq i, j \leq m - 1 = 2t - 1 $ with $ 2 \leq i + j \leq m $. We denote the left (resp. right) vertex set $ \{1, 2, \ldots, m\} $ (resp. $ \{1', 2', \ldots, m'\} $) of $ (\Gamma_0, \omega_0) $ by $ I_0 $ (resp. $ J_0 $).
\subsection{Proof of $ x(\Gamma_0, \omega_0) = \mathsf{M} $}

In this subsection, we use the following notation.
\begin{itemize}
  \item Pochhammer symbol.
        \begin{equation*}
          (z)_s \coloneqq \frac{\Gamma(z + s)}{\Gamma(z)} \ (s \in \mathbb{Z}_{\geq 0}),
        \end{equation*}
        where $ \Gamma(z) $ is the Gamma function.
  \item Generalized hypergeometric function.
        \begin{equation*}
          {}_pF_q\left(\begin{matrix}
              a_1, \, a_2, \, \ldots, \, a_p \\
              b_1, b_2, \ldots, b_q          \\
            \end{matrix}
          \bigg| \, z\right) \coloneqq \sum_{s = 0}^{\infty}\frac{(a_1)_s (a_2)_s \cdots (a_p)_s}{(b_1)_s (b_2)_s \cdots (b_q)_s}\frac{z^s}{s!}.
        \end{equation*}
  \item The subplanar network $ R_{a, b}(r_{i, j}) \ (1 \leq i, j \leq m - 1 )$ of $ (\Gamma_0, \omega_0) $ is a weighted $ a \times b $ square lattice cut out from $ (\Gamma_0, \omega_0) $ such that it consists only of blue edges whose weights are of the form $ r_{i', j'} = \frac{i' + (t - j')}{i' + j' - 1} $ and black edges, and its top-right blue edge has weight $ r_{i, j} $.
  \item The planar network $ R_{a, b}'(i + (t - j)) $ is a weighted $ a \times b $ square lattice defined by replacing the weights $ \frac{i' + (t - j')}{i' + j' - 1} $ of all blue edges in $ R_{a, b}(r_{i, j}) $ by $ i' + (t - j') $.
  \item $ w_{a, b}(k) $ is the sum of the weights of all lattice paths from the bottom-left corner of $ R_{a, b}'(k) $ to its top-right corner.
  \item Let $ i \in I_0 $ and $ j \in J_0 $. $ w(i \nearrow j) $ (resp. $ w(i \searrow j) $) is the sum of the weights of all paths in $ (\Gamma_0, \omega_0) $ from the vertex $ i $ to $ j $ with no downward (resp. upward) steps. $ w(i \to j) $ is the sum of the weights of all paths from $ i $ to $ j $.
\end{itemize}
\begin{lem}\label{lem_weight}
  \begin{equation}\label{eq_weight}
    w_{a, b}(k) = \binom{a + b}{a}(k + b)_a.
  \end{equation}
\end{lem}
\begin{proof}
  We prove \eqref{eq_weight} by induction on $ b $.
  If $ b = 0 $, we have $ w_{a, 0}(k) = (k)_{a} = \binom{a + 0}{a}(k + 0)_a $, so the base case holds. Assume \eqref{eq_weight} holds for $ b - 1 $: $ w_{a, b - 1}(k) = \binom{a + (b - 1)}{a}(k + (b - 1))_a $. Then, by classifying the lattice paths from the bottom-left corner to the top-right corner by the position of their first horizontal step, we obtain
  \begin{align*}
    w_{a, b}(k) & = \sum_{s = 0}^a (k + a + 2b - s)_s w_{a - s, b - 1}(k)                                                                                                                                                      \\
                & = \sum_{s = 0}^a \binom{a + b - s - 1}{a - s}(k + a + 2b - s)_s(k + b - 1)_{a - s}                                                                                                                           \\
                & = \sum_{s = 0}^a \frac{\binom{a}{s}\binom{a + b - 1}{a}}{\binom{a + b - 1}{s}} \frac{\binom{k + a + 2b - 1}{s}\binom{k + a + b - 2}{a}}{\binom{k + a + b - 2}{s}}a!                                          \\
                & = \sum_{s = 0}^a \frac{\binom{a}{s}\binom{a + b - 1}{a}}{\binom{a + b - 1}{s}} \frac{\binom{k + a + 2b - 1}{s}\binom{k + a + b - 2}{a}}{\binom{k + a + b - 2}{s}}\frac{(k + b)_a}{\binom{k + a + b - 1}{a}}.
  \end{align*}
  Hence, it is enough to show that
  \begin{equation*}
    \frac{\binom{k + a + b - 1}{a}\binom{a + b}{a}}{\binom{k + a + b - 2}{a}\binom{a + b - 1}{a}} = \sum_{s = 0}^a\frac{\binom{a}{s}\binom{k + a + 2b - 1}{s}}{\binom{a + b - 1}{s}\binom{k + a + b - 2}{s}}.
  \end{equation*}
  The left-hand side can be written as $ \frac{(k + a + b - 1)(a + b)}{(k + b - 1)b} $. On the other hand, since $ \binom{A}{s} = (-1)^s\frac{(-A)_s}{s!} $, the right-hand side can be expressed as
  \begin{equation*}
    \sum_{s = 0}^a \frac{(1)_s(-k - a - 2b + 1)_s(-a)_s}{(- a - b + 1)_s(-k - a - b + 2)_s}\frac{1}{s!} =
    {}_3F_2\left(
    \begin{matrix}
        1, \, - k - a - 2b + 1, \, -a   \\
        - a - b + 1, \, - k - a - b + 2 \\
      \end{matrix}
    \biggm|1\right).
  \end{equation*}
  By Saalsch\"utz's formula (W. N. Bailey \cite{Bailey}, Chapter II, p.~9), this quantity is equal to
  \begin{equation*}
    \frac{(-a - b)_a(k + b)_a}{(-a - b + 1)_a(k + b - 1)_a} = \frac{(k + a + b - 1)(a + b)}{(k + b - 1)b}.
  \end{equation*}
\end{proof}
\begin{lem}\label{lem_up}
  For $ 1 \leq i \leq t $ and $ 1 \leq j \leq m $, we have
  \begin{equation*}
    w(i \nearrow j) = \binom{t}{j - i}.
  \end{equation*}
\end{lem}
\begin{proof}
  If $ j \leq i $, the assertion is trivial.
  Thus, we may assume $ j > i $. First, consider the case $ j \leq t + 1 $. Computing $ w(i \nearrow j) $, we use the square lattice $  R_{j - i, i - 1}(r_{1, j - 1}) = R_{j - i, i - 1}(\frac{t - j + 2}{j - 1}) $. Applying Lemma~\ref{lem_weight} to $ R_{j - i, i - 1}'(t - j + 2) $, we have
  \begin{equation*}
    w(i \nearrow j) = \frac{w_{j - i, i - 1}(t - j + 2)}{(i)_{j - i}} = \binom{j - 1}{j - i}\frac{((t - j + 2) + (i - 1))_{j - i}}{(i)_{j - i}} = \binom{t}{j - i}.
  \end{equation*}
  Next, consider the case $ j > t $. Then the square lattice under consideration is $ R_{j - i, t - (j - i)}(r_{j - t, t}) = R_{j - i, t - (j - i)}(\frac{j - t}{j - 1}) $. Applying Lemma~\ref{lem_weight} to $ R_{j - i, t - (j - i)}'(j - t) $, we have
  \begin{equation*}
    w(i \nearrow j) = \frac{w_{j - i, t - (j - i)}(j - t)}{(i)_{j - i}} = \binom{t}{j - i}\frac{((j - t) + (t - (j - i)))_{j - i}}{(i)_{j - i}} = \binom{t}{j - i}.
  \end{equation*}
\end{proof}
\begin{lem}\label{lem_weight_1}
  For $ 1 \leq i \leq t $ and $ 1 \leq j \leq m $, we have
  \begin{equation*}
    w(i \to j) = \binom{t + i - 1}{j - 1}.
  \end{equation*}
\end{lem}
\begin{proof}
  Note that $ w(i \searrow k) = \binom{i - 1}{k - 1}$. It follows from the definition of the weights $ l_{i', j'} $. Thus, from Lemma~\ref{lem_up} and Vandermonde's convolution formula, we have
  \begin{equation*}
    w(i \to j) = \sum_{k = 1}^i w(i \searrow k)w(k \nearrow j) = \sum_{k = 1}^i \binom{i - 1}{k - 1}\binom{t}{j - k} = \binom{t + i - 1}{j - 1}.
  \end{equation*}
\end{proof}
\begin{lem}\label{lem_weight_2}
  For $ t + 1 \leq i \leq m $ and $ 1 \leq j \leq m $, we have
  \begin{equation*}
    w(i \to j) = \binom{2t - i}{j - i}.
  \end{equation*}
\end{lem}
\begin{proof}
  Since $ l_{i', j'} = 0 $ for $ i' + j' > t $, we have $ w(i \to j) = w(i \nearrow j) $ and we may assume $ j > i $. Each lattice path from $ i $ to $ j $ is classified by the smallest nonnegative integer $ 0 \leq s \leq j - i $ such that the path passes through the bottom-left corner of the planar network $ R_{j - i - s, 2t - j}(r_{j - t, t}) = R_{j - i - s, 2t - j}(\frac{j - t}{j - 1}) $. Hence, we obtain
  \begin{align*}
    w(i \nearrow j) & = \sum_{s = 0}^{j - i}\left\{\binom{i - t - 1 + s}{i - t - 1}(2t - i - s + 1)_s \times \frac{w_{j - i - s, 2t - j}(j - t)}{(i)_{j - i}}\right\}            \\
                    & = \sum_{s = 0}^{j - i}\binom{i - t - 1 + s}{i - t - 1}(2t - i - s + 1)_s \frac{\binom{2t - i - s}{j - i - s}((j - t) + (2t - j))_{j - i - s}}{(i)_{j - i}} \\
                    & = \sum_{s = 0}^{j - i}\binom{i - t - 1 + s}{i - t - 1}\binom{t - i + j - 1 - s}{t - 1}\frac{\binom{2t - 1}{j - 1}}{\binom{2t - 1}{i - 1}}.
  \end{align*}
  Thus, it suffices to show that
  \begin{equation*}
    \binom{j - 1}{i - 1} = \frac{\binom{2t - i}{j - i}\binom{2t - 1}{i - 1}}{\binom{2t - 1}{j - 1}} = \sum_{s = 0}^{j - i}\binom{i - t - 1 + s}{i - t - 1}\binom{t - i + j - 1 - s}{t - 1}.
  \end{equation*}
  Using the relation $ \binom{A + B}{B} = (-1)^B\binom{- A - 1}{B} $, the right-hand side of the above equation is represented as
  \begin{equation*}
    \sum_{s = 0}^{j - i}(-1)^{j - i}\binom{-i + t}{s}\binom{-t}{j - i - s},
  \end{equation*}
  and is equal to $ (-1)^{j - i}\binom{-i}{j - i} = \binom{j - 1}{i - 1} $ by Chu--Vandermonde's convolution formula.
\end{proof}
From Lemmata~\ref{lem_weight_1} and~\ref{lem_weight_2}, we obtain the desired theorem.
\begin{thm}\label{thm_M}
  Define the edge weight function $ \omega_0 $ on $ \Gamma_0 $ as in Subsection~\ref{subsec_construction_graph}. Then we have
  \begin{equation*}
    x(\Gamma_0, \omega_0) = \mathsf{M}.
  \end{equation*}
\end{thm}
As observed in Subsection~\ref{subsec_observation}, this implies the Fujimoto conjecture.
\begin{cor}[Fujimoto conjecture, Theorem~\ref{thm_main_theorem}, restated]\label{cor_Fujimoto}
  The Fujimoto conjecture (Conjecture~\ref{conj_Fujimoto}) is true for all even integers $ m \geq 2 $.
\end{cor}
\begin{proof}
  It is enough to show that any $ m \times m $ submatrix $ M' $ of $ \mathsf{M}' $ is invertible. If $ u $ rows are chosen from the first $ t $ rows of $ \mathsf{M}' $, we have
  \begin{equation*}
    M' \sim
    \begin{pmatrix}
      I_u & O \\
      O   & M \\
    \end{pmatrix},
  \end{equation*}
  where $ M $ is a submatrix of $ \mathsf{M} $ and $ \sim $ means that the two matrices are equivalent under elementary row and column operations. If we write $ M $ as $ \mathsf{M}_{I, J} \ (I \subset I_0, J \subset J_0)$, clearly $ \{(t + 1)', (t + 2)', \ldots, m'\} \subset J $. Thus, by the definition of the edge weight function $ \omega_0 $, there exists at least one collection of non-intersecting paths from the vertex set $ I $ to $ J $ in $ (\Gamma_0, \omega_0) $ such that all paths have positive weights. Therefore, by Theorem~\ref{thm_LGV} and Theorem~\ref{thm_M}, we conclude that $ \det M > 0 $, and thus $ \det M' \neq 0 $.
\end{proof}

\section*{Questions}
Chern and Osserman \cite{Chern-Osserman} proved that for algebraic minimal surfaces (i.e., complete minimal surfaces with finite total curvature) immersed in $ \mathbb{R}^m $, if the Gauss map is non-degenerate, then it omits at most $ \frac{(m - 1)(m + 2)}{2} $ hyperplanes in $ \mathbb{P}^{m - 1} $ in general position. Ru \cite{Ru_2} later showed that the non-degeneracy assumption can be removed. Since $ \frac{m(m + 1)}{2} - \frac{(m - 1)(m + 2)}{2} = 1 > 0 $,
there does not exist an \emph{algebraic} minimal surface $ S $ in $ \mathbb{R}^m $ whose Gauss map omits $ \frac{m(m + 1)}{2} $ hyperplanes in $ \mathbb{P}^{m - 1} $ in general position.
\begin{ques}
  What is the optimal number of omitted hyperplanes in $ \mathbb{P}^{m - 1} $ in general position for the generalized Gauss map of algebraic minimal surfaces in $ \mathbb{R}^m $?
\end{ques}

For $ m = 3 $, the optimal number is expected to be $ 3 $ (equivalently, the classical Gauss map omits at most $ 2 $ points), known as the \textit{Osserman conjecture}. In R. Kobayashi's report \cite{Kobayashi-Miyaoka} (joint with R. Miyaoka), it is suggested that this conjecture may be settled by establishing a ``Second Main Theorem'' within the framework of ``Nevanlinna--Galois theory''.

\vspace{+0.5mm}
In Nevanlinna theory, the geometric meaning of the optimal number of omitted values of holomorphic curves is well known. For example, by the theory of covering surfaces, the number $ 2 $, which appears in Picard's theorem as the optimal number of omitted values, can be interpreted as the Euler characteristic of the Riemann sphere $ \mathbb{P}^1 \simeq \mathbb{C} \cup \{\infty\} $.
\begin{ques}
  What is the geometric meaning of the number $ \frac{m(m + 1)}{2} $?
\end{ques}

\appendix
\section*{Appendix}
For the convenience of the reader, we give some auxiliary functions used in the construction of best possible examples for Theorems~\ref{thm_gauss_map_Rm} and~\ref{thm_gauss_map_Cm}, following Fujimoto~\cite{Fujimoto_3}, and construct a best possible example for Theorem~\ref{thm_gauss_map_Rm}. We denote the \emph{Wronskian} of $ m $ holomorphic functions $ h_1(z), h_2(z), \ldots, h_m(z) $ by $ W(h_1, h_2, \ldots, h_m) $. We use the following property of the Wronskian:
\begin{center}
  $ h_1(z), h_2(z), \ldots, h_m(z) $ are linearly independent over $ \mathbb{C} $

  if and only if $ W(h_1, h_2, \ldots, h_m)(z) $ is not identically zero.
\end{center}

\begin{lem*}(cf. \cite{Fujimoto_3})
  For each even integer $ m = 2t \geq 2 $, consider the following $ 3t + m(t - 1) = \frac{m(m + 1)}{2} $ polynomials
  \begin{alignat*}{3}
     & f_i(z) \ (\text{as defined in Conjecture~\ref{conj_Fujimoto}})         &       & (1 \leq i \leq 3t), \\
     & f_{3t + i}(z) \coloneqq (z - a_2)^{m - i}(z - b_2)^{i - 1}             &       & (1 \leq i \leq m),  \\
     & \quad \vdots                                                           &       &                     \\
     & f_{3t + 2t(j - 2) + i}(z) \coloneqq (z - a_j)^{m - i}(z - b_j)^{i - 1} &       & (1 \leq i \leq m),  \\
     & \quad \vdots                                                           &       &                     \\
     & f_{3t + 2t(t - 2) + i}(z) \coloneqq (z - a_t)^{m - i}(z - b_t)^{i - 1} & \quad & (1 \leq i \leq m),  \\
  \end{alignat*}
  where $ a_1 \coloneqq 0, b_1 \coloneqq 1, a_j, b_j \ (2 \leq j \leq t) $ are distinct complex numbers. If these numbers are chosen suitably, then any $ m $ polynomials chosen from the above are linearly independent over $ \mathbb{C} $.
\end{lem*}
\begin{proof}
  Fix $ t \geq 1 $.
  We claim that for each $ 2 \leq u \leq t + 1 $, any $ m $ polynomials chosen from $ f_i(z) \ (1 \leq i \leq 3t + 2t(u - 2)) $ are linearly independent, provided that the complex numbers $ a_2, b_2, \ldots, a_{u - 1}, b_{u - 1} $ are chosen suitably. Our goal is to prove this for the case $ u = t + 1 $. We prove this claim by induction on $ u $. The base case $ u = 2 $ follows from Corollary~\ref{cor_Fujimoto}. Assume that the claim holds for $ u - 1 \ (u > 2) $. By the induction hypothesis, there exist distinct complex numbers $ a_2, b_2, \ldots, a_{u - 2}, b_{u - 2} $ such that any $ m $ polynomials chosen from $ f_i(z) \ (1 \leq i \leq 3t + 2t(u - 3)) $ are linearly independent. Hence, we may assume that the $ m $ polynomials $ f_{i_1}, f_{i_2}, \ldots, f_{i_m} $ under consideration are labelled by the following indices
  \begin{equation*}
    i_1 < i_2 < \cdots < i_k \leq 3t + 2t(u - 3) < i_{k + 1} < \cdots < i_m.
  \end{equation*}
  Set $ g_r(z) \coloneqq f_{i_r}(z) \ (1 \leq r \leq m) $. Note that $ W(g_1, g_2, \ldots, g_k)(z) $ is not identically zero. So, we can find a complex number $ c $ such that $ W(g_1, g_2, \ldots, g_k)(c) \neq 0 $. By a coordinate translation, we may assume $ W(g_1, g_2, \ldots, g_k) (0) \neq 0 $. We remark that the numbers $ a_1, b_1, \ldots, a_{u - 2}, b_{u - 2} $ are also translated in this process. If one of the values is $ 0 $, we choose another $ c $. We write the $ m $ polynomials $ g_r $ by
  \begin{equation*}
    g_r(z) = \sum_{s = 0}^{m - 1}A_{r, s + 1}z^s,
  \end{equation*}
  where each $ A_{r, s} \ (1 \leq r, s \leq m) $ is a polynomial in $ a_{u - 1} $ and $ b_{u - 1} $. We want to show that $ F_{i_1, i_2, \ldots, i_m}(a_{u - 1}) \coloneqq \det (A_{r, s})_{1 \leq r, s \leq m} $ is nonzero as a polynomial in $ a_{u - 1} $ by choosing $ b_{u - 1} $ suitably.
  By setting $ l_r \coloneqq m - (i_r - 3t - 2t(u - 3)) \geq 0 $, we can write $ f_{i_r} = f_{3t + 2t(u - 3) + (m - l_r)} $.
  Now set $ b_{u - 1} \coloneqq 0 $. For $ r \geq k + 1 $, we have
  \begin{align*}
    g_r(z) = (z - a_{u - 1})^{l_r}z^{m - l_r - 1} = \sum_{s = 0}^{l_r}\binom{l_r}{s}(-a_{u - 1})^{s}z^{m - s - 1} = \sum_{s = 0}^{m - 1}\binom{l_r}{m - s - 1}(-a_{u - 1})^{m - s - 1}z^s.
  \end{align*}
  Hence, we have
  \begin{equation*}
    A_{r, s + 1} = \binom{l_r}{m - s - 1}(-a_{u - 1})^{m - s - 1} \quad (r \geq k + 1).
  \end{equation*}
  Note that $ s + 1 < m - l_r $ if and only if $ A_{r, s + 1} = 0 $ as a polynomial in $ a_{u - 1} $ for $ r \geq k + 1 $. In particular, since
  \begin{equation*}
    m - l_r = i_r - 3t - 2t(u - 3) < i_{r + 1} - 3t - 2t(u - 3) < \cdots < i_{m} - 3t - 2t(u - 3) \leq m,
  \end{equation*}
  the rightmost $ m - r + 1 $ entries of the $ r $-th ($ r \geq k+ 1 $) row of the matrix $ (A_{r, s})_{1 \leq r, s \leq m} $ are nonzero monomials, and $ \deg A_{r, s} = m - s $. By the Laplace expansion of the determinant along the first $ k $ columns and the last $ m - k $ columns of $ (A_{r, s})_{1 \leq r, s \leq m} $, since the entries in the first $ k $ rows of this matrix are independent of $ a_{u - 1} $, it follows that
  \begin{align*}
    \deg F_{i_1, i_2, \ldots, i_m}(a_{u - 1}) & \leq \deg \det(A_{r, s})_{k + 1 \leq r, s \leq m}                                     \\
                                              & \leq (m - (k + 1)) + (m - (k + 2)) + \cdots + (m - m) = \frac{(m - k - 1)(m - k)}{2},
  \end{align*}
  where equality holds in both of the above inequalities if $ \det(A_{r, s})_{1 \leq r, s \leq k} \cdot \det(A_{r, s})_{k + 1 \leq r, s \leq m} \neq 0 $. We now verify that this condition is indeed valid. Since
  \begin{equation*}
    \left. \frac{d^{v}}{dz^{v}}g_r(z) \right|_{z = 0} = \left. \frac{d^{v}}{dz^{v}} \left(\sum_{s = 0}^{m - 1}A_{r, s + 1}z^s\right) \right|_{z = 0} = v! A_{r, v + 1},
  \end{equation*}
  we have $ C \times \det(A_{r, s})_{1 \leq r, s \leq k} = W(g_1, g_2, \ldots, g_k)(0) \neq 0 $, where $ C $ is a nonzero constant. On the other hand,
  \begin{align*}
    \det (A_{r, s})_{k + 1 \leq r, s \leq m}
     & = \det \left((-1)^{m - s}\binom{l_r}{m - s}\right)_{k + 1 \leq r, s \leq m}a_{u - 1}^{\frac{(m - k - 1)(m - k)}{2}} \\
     & = \left(C' \times \det(l_r^{m - s})_{k + 1 \leq r, s \leq m}\right)a_{u - 1}^{\frac{(m - k - 1)(m - k)}{2}} \neq 0
  \end{align*}
  since the Vandermonde determinant does not vanish.
  Thus, $ F_{i_1, i_2, \ldots, i_m}(a_{u - 1}) $ is nonzero as a polynomial. In a similar way, for other indices $ 1 \leq i_1' < i_2' < \cdots < i_m' \leq 3t + 2t(u - 2) $, we can show that $ F_{i_1', i_2', \ldots, i_m'}(a_{u - 1}) $ is also nonzero as a polynomial. (Note that we can take a common $ c $ for all these indices in the above proof, and thus the numbers $ a_1, b_1, a_2, b_2, \ldots, a_{u - 2}, b_{u - 2} $ and $ b_{u - 1} $ can be chosen independently of the indices.) Since there are only finitely many polynomials $ F_{i_1, i_2, \ldots, i_m}(a_{u - 1}) $, we can take a complex number $ a_{u - 1} $ such that it differs from $ a_1, b_1, a_2, b_2, \ldots, a_{u - 2}, b_{u - 2}, b_{u - 1} $, and none of these polynomials vanish at $ a_{u - 1} $. By translating the coordinate again so that $ a_1 = 0 $, the proof is complete.
\end{proof}

Using this lemma, we give a best possible example for Theorem~\ref{thm_gauss_map_Rm}. Here we construct the minimal surface so that it is pseudo-algebraic, following the argument in \cite{Fujimoto_1}, pp.~42--43, \cite{Fujimoto_7}, pp.~198--199, and \cite{Kawakami}, pp.45--47.

\begin{thm*}
  Let $ m \geq 4 $ be an even integer. Then there exists a pseudo-algebraic minimal surface in $ \mathbb{R}^m $ whose Gauss map is non-degenerate and omits $ \frac{m(m + 1)}{2} $ hyperplanes in $ \mathbb{P}^{m - 1} $ in general position.
\end{thm*}

\begin{proof}
  Let $ t \coloneqq \frac{m}{2} $. We define $ m $ functions
  \begin{alignat*}{3}
     & h_{2l + 1}(z) \coloneqq z^l + z^{2t - l - 1}              & \quad & \qquad (0 \leq l \leq t - 2),                               \\
     & h_{2l + 2}(z) \coloneqq \sqrt{-1}(z^l - z^{2t - l - 1})   & \quad & \qquad (0 \leq l \leq t - 2),                               \\
     & h_{2t - 1}(z) \coloneqq \sqrt{-(t - 1)}(z^{t - 1} + z^t), & \quad & h_m(z) = h_{2t}(z) \coloneqq \sqrt{t - 1}(z^{t - 1} - z^t).
  \end{alignat*}
  Then we have
  \begin{equation*}
    \sum_{k = 1}^m h_k(z)^2 = \sum_{l = 0}^{t - 2}\{(z^l + z^{2t - l - 1})^2 - (z^l - z^{2t - l - 1})^2\} -(t - 1)\{(z^{t - 1} + z^t)^2 - (z^{t - 1} - z^t)^2 \} = 0.
  \end{equation*}
  By the above lemma, for suitable constants $ a_1 \coloneqq 0, b_1 \coloneqq 1 $, and $ a_j, b_j \ (2 \leq j \leq t) $, any $ m $ polynomials among $ f_j(z) \ (1 \leq j \leq \frac{m(m + 1)}{2}) $ are linearly independent over $ \mathbb{C} $. Now set
  \begin{equation*}
    S \coloneqq \mathbb{C} \setminus \{a_1, a_2, \ldots, a_t, b_1, b_2, \ldots, b_t\},
  \end{equation*}
  and let $ \pi : \widetilde{S} \to S $ denote the universal covering. Using the function
  \begin{equation*}
    \psi(z) \coloneqq \frac{1}{(z - a_1)(z - b_1)(z - a_2)(z - b_2) \cdots (z - a_t)(z - b_t)},
  \end{equation*}
  we define $ m $ holomorphic functions $ \widetilde{g}_k \coloneqq \psi h_k \ (1 \leq k \leq m) $ on $ S $ and holomorphic 1-forms $ \phi_k \coloneqq (\widetilde{g}_k \circ \pi)dz \ (1 \leq k \leq m) $ on $ \widetilde{S} $. Then $ \phi_1, \phi_2, \ldots, \phi_m $ have no common zeros. Next, we define $ m $ functions $ x_i $ on $ \widetilde{S} $ by
  \begin{equation*}
    x_i(z) \coloneqq \Re \int_{z_0}^{z}\phi_i \quad (1 \leq i \leq m).
  \end{equation*}
  Since $ \sum_{k = 1}^m(\widetilde{g}_k \circ \pi)^2 = 0 $, by the Enneper--Weierstrass representation (see \cite{Kawakami}, Theorem~2.9, p.~8), the surface
  \begin{equation*}
    x \coloneqq (x_1, x_2, \ldots, x_m) : \widetilde{S} \to \mathbb{R}^m
  \end{equation*}
  is a minimal surface. The Gauss map is given by $ G = (\widetilde{g}_1 \circ \pi : \widetilde{g}_2 \circ \pi : \cdots : \widetilde{g}_m \circ \pi ) $, and hence $ G = (h_1 \circ \pi : h_2 \circ \pi : \cdots : h_m \circ \pi) $. Since $ h_1, h_2, \ldots, h_m $ are linearly independent over $ \mathbb{C} $, the Gauss map $ G $ is non-degenerate. Moreover, since $ h_1, h_2, \ldots, h_m $ form a basis of the space of all polynomials of degree at most $ m - 1 $, we can write
  \begin{equation*}
    f_i(z) = \sum_{j = 1}^m c_{ij}h_{j}(z)
  \end{equation*}
  for some constants $ c_{ij} \ (1 \leq i \leq \frac{m(m + 1)}{2}, 1 \leq j \leq m) $. We consider $ \frac{m(m + 1)}{2} $ hyperplanes
  \begin{equation*}
    H_i : c_{i1}w_1 + c_{i2}w_2 + \cdots + c_{im}w_m = 0
  \end{equation*}
  in $ \mathbb{P}^{m - 1} $. Note that these hyperplanes are in general position. Moreover, the Gauss map $ G $ omits these hyperplanes. Indeed, for $ 1 \leq i \leq \frac{m(m + 1)}{2} $, we have
  \begin{equation*}
    \sum_{j = 1}^m c_{ij}(h_j \circ \pi)(z) = f_i(\pi(z)) = (\pi(z) - a_i)^{r_i}(\pi(z) - b_i)^{s_i}
  \end{equation*}
  for some $ r_i $ and $ s_i $, and these functions do not vanish by the definition of $ S $. To complete the proof, it is enough to show that the minimal surface $ \widetilde{S} $ is complete. Since the metric $ d\widetilde{s}^2 $ on $ \widetilde{S} $ is induced from the metric $ d{s}^2 $ on $ S $ via $ \pi $, it suffices to verify the completeness of $ S $. The metric $ d{s}^2 $ induced by the standard Euclidean metric on $ \mathbb{R}^m $ is given by
  \begin{align*}
    d{s}^2 & = \frac{1}{2}(|\widetilde{g}_1(z)|^2 + |\widetilde{g}_2(z)|^2 + \cdots + |\widetilde{g}_m(z)|^2)|dz|^2                                                                        \\
           & = \frac{\sum_{l = 0}^{t - 2}(|z|^{2l} + |z|^{2(2t - l - 1)}) + (t - 1)(|z|^{2(t - 1)} + |z|^{2t})}{|(z - a_1)(z - b_1)(z - a_2)(z - b_2) \cdots (z - a_t)(z - b_t)|^2}|dz|^2.
  \end{align*}
  Around $ \infty $, taking the local coordinate $ \zeta = \frac{1}{z} $, it is represented as
  \begin{equation*}
    d{s}^2 = \frac{\sum_{l = 0}^{t - 2}(|\zeta|^{2(2t - l - 1)} + |\zeta|^{2l}) + (t - 1)(|\zeta|^{2t} + |\zeta|^{2(t - 1)})}{|(1 - b_1\zeta)(1 - a_2\zeta)(1 - b_2\zeta) \cdots (1 - a_t\zeta)(1 - b_t\zeta)|^2}\frac{|d\zeta|^2}{|\zeta|^2}.
  \end{equation*}
  We take a piecewise smooth curve $ \gamma(t) \ (0 \leq t < 1) $ in $ S $ tending to a boundary point of $ S $, namely one of $ a_1, b_1, a_2, b_2, \ldots, a_t, b_t, \infty $, as $ t \to 1 $. By the above expressions for $ d{s}^2 $, we can easily verify that the length of $ \gamma $ is infinite. Therefore, by construction, $ x : \widetilde{S} \to \mathbb{R}^{m} $ is a pseudo-algebraic minimal surface. This completes the proof.
\end{proof}
\section*{Acknowledgments.}
The author thanks Professors Ryoichi Kobayashi and Sho Tanimoto for helpful comments and suggestions. The author also thanks Ryotaro Nishimura for daily discussions.

In preparing the figures, the author referred to the \TeX\ source code of \cite{Fomin-Zelevinsky}.


\begin{thebibliography}{99}
  \bibitem{Bailey} W. N. Bailey, \textit{Generalized hypergeometric series}, Cambridge Math. Tracts \textbf{32}, Cambridge Univ. Press, 1935.
  \bibitem{Chern-Osserman} S. S. Chern and R. Osserman, Complete minimal surfaces in euclidean $ n $-space, J. Anal. Math. \textbf{19} (1967), pp.~15--34.
  \bibitem{Fomin-Zelevinsky} S. Fomin and A. Zelevinsky, \textit{Total positivity: Tests and parametrizations}, The Mathematical Math. Intelligencer \textbf{22} (2000), pp.~23--33.
  \bibitem{Fujimoto_2} H. Fujimoto, \textit{On the value distribution of the Gauss map of minimal surfaces in $ \mathbb{R}^m $} (in Japanese)[Translated from Japanese], SUGAKU \textbf{40} (1988), pp.~312--321.
  \bibitem{Fujimoto_1} H. Fujimoto, \textit{Examples of complete minimal surfaces in $ \mathbb{R}^m $ whose Gauss map omit at most $ \frac{m(m + 1)}{2} $ hyperplanes in general position}, Sci. Rep. Kanazawa Univ. \textbf{33} (1989), pp.~37--43.
  \bibitem{Fujimoto_3} H. Fujimoto, \textit{Modified defect relations for the Gauss map of minimal surfaces. II}, J. Differential Geom., \textbf{31} (1990), pp.~365--385.
  \bibitem{Fujimoto_4} H. Fujimoto, \textit{Modified defect relations for the Gauss map of minimal surfaces, III}, Nagoya math. J. \textbf{124} (1991), pp.~13--40.
  \bibitem{Fujimoto_5} H. Fujimoto, \textit{Some function-theoretic properties of the Gauss map of minimal surfaces} \ (Holomorphic mappings, Diophantine Geometry and Related Topics : in Honor of Professor Shoshichi Kobayashi on his 60th Birthday), \textbf{819} (1993), pp.~35--49.
  \bibitem{Fujimoto_6} H. Fujimoto, \textit{Gauss Maps of Complete Minimal Surfaces} In : Progress in Differential Geometry, Adv. Stud. Pure Math. \textbf{22} (1993). pp.~13--29.
  \bibitem{Fujimoto_7} H. Fujimoto, \textit{Value Distribution Theory of the Gauss Map of Minimal Surfaces in $\mathbb{R}^m$}, Aspects of Math. \textbf{E21}, 1993.
  \bibitem{Gessel-Viennot} I. Gessel and G. X. Viennot, \textit{Binomial determinants, paths, and hooklength formulae}, Adv. in Math. \textbf{58} (1985), pp.~300--321.
  \bibitem{Kawakami} Y. Kawakami, \textit{Value distribution theoretical properties of the Gauss map of pseudo-algebraic minimal surfaces}, 2006, arXiv:math/0608351, preprint.
  \bibitem{Kobayashi-Miyaoka} R. Kobayashi (in collaboration with R. Miyaoka), \textit{Collective Cohn--Vossen Problem, Parabolic Localization Principle and Period Condition of Algebraic Minimal Surfaces}, report at the Eighteenth Oka Symposium, 2020.
  \bibitem{Lindstrom} B. Lindstr\"om, \textit{On the vector representations of induced matroids}, Bull. London Math. Soc. \textbf{5} (1973), pp.~85--90.
  \bibitem{Osserman} R. Osserman, \textit{A Survey of Minimal Surfaces}, 2nd ed., Dover Publications Inc., 1986.
  \bibitem{Osserman-Ru} R. Osserman and M. Ru, \textit{An estimate for the Gauss map curvature of minimal surfaces in $ \mathbb{R}^m $ whose Gauss map omits a set of hyperplanes}, J. Differential Geom., \textbf{45} (1997), pp.~578--593.
  \bibitem{Ru} M. Ru, \textit{On the Gauss map of Minimal surfaces immersed in $ \mathbb{R}^n $}, J. Differential Geom. \textbf{34} (1991), pp.~411--423.
  \bibitem{Ru_2} M. Ru, \textit{On the Gauss map of minimal surfaces with finite total curvature}, Bull. Austral. Math. Soc. \textbf{44} (1991), pp.~225--232.
  \bibitem{Ru_3} M. Ru, Minimal surfaces through Nevanlinna theory, De Gruyter Stud. Math. \textbf{92}, 2023.
  \bibitem{Whitney} A. M. Whitney, A reduction theorem for totally positive matrices, J. d'Analyse Math. \textbf{2} (1952), pp.~88--92.
\end{thebibliography}
\end{document}